\documentclass[10pt]{amsart}
\usepackage{amssymb,latexsym,amsmath,amsthm,enumitem,mathrsfs,geometry}
\usepackage{fullpage}
\usepackage{fancyhdr}

\usepackage{comment}
\usepackage{mathrsfs}
\usepackage{hyperref}
\usepackage{lineno}
\usepackage{diagbox}
\usepackage{array}
\usepackage{tikz}
\usetikzlibrary{arrows}
\usetikzlibrary{positioning}
\usepackage{float}
\newtheorem{thm}{Theorem}[section]
\newtheorem*{thm*}{Theorem}

\newtheorem{lemma}[thm]{Lemma}

\newcommand{\beq}{\begin{equation}}
\newcommand{\eeq}{\end{equation}}

\newtheorem{remark}{Remark}

\usepackage[titletoc,title]{appendix}

\newcommand{\N}{\mathbb{N}}

\def\O{\operatorname{O}}
\def\o{\operatorname{o}}

\newcommand{\kommentar}[1]{}
\newtheorem{theorem}{Theorem}%[section]

\newtheorem*{remark*}{Remark}

\definecolor{pink}{rgb}{1,.2,.6}
\definecolor{orange}{rgb}{0.7,0.3,0}
\definecolor{blue}{rgb}{.2,.6,.75}
\definecolor{green}{rgb}{.4,.7,.4}
\definecolor{purple}{RGB}{127,0,255}

\begin{document}
\numberwithin{equation}{section}

\title{A lower bound for the discrepancy in a  Sato-Tate type measure}

\author[Das]{Jishu Das}
\address{Indian Institute of Science Education and Research Pune, Dr. Homi Bhabha Road, Pune 411008}
\email{jishu.das@students.iiserpune.ac.in}

\thanks{}

\date{\today}

\begin{abstract} 
Let $S_k(N)$ denote the space of cusp forms of even integer weight $k$ and level $N$. We prove an asymptotic for the
Petersson trace formula for $S_k(N)$ under an appropriate condition.
Using the non-vanishing of a Kloosterman sum involved in the asymptotic, we give a lower bound for discrepancy in the Sato-Tate distribution for levels not divisible by $8$. 
This generalizes a result of Jung and Sardari  \cite[Theorem 1.6]{JS} for squarefree levels. 
An analogue of the Sato-Tate distribution was obtained by Omar and Mazhouda \cite[Theorem 3]{OM} for the distribution of eigenvalues $\lambda_{p^2}(f)$ where $f$ is a Hecke eigenform and  $p$ is a prime number. 
%Using the Petersson trace formula, with $\gcd(p, N) = 1$ and  $k+N\rightarrow \infty$, it can be shown
%that a measure $\nu_{k,N,2}$ converges weakly to $\mu_{\infty,2},$ an analogue of the Sato-Tate measure that corresponds to the distribution of eigenvalues  $\lambda_{p^2}(f)$ for an eigenform $f$.
As an application of the above-mentioned asymptotic, we obtain a sequence of weights $k_n$ such that  discrepancy in the analogue distribution obtained in \cite{OM}
 has a lower bound.
\end{abstract}
\maketitle 

\section{Introduction}
Let $S_k(N)$ denote the space of cusp forms of even integer weight $k$ and level $N$. For $(n,N)=1,$ the $n^{\text{th}}$ normalised Hecke operator acting on $S_k(N)$ is given by $$T_n(f)(z)=n^{\frac{k-1}{2}}\sum_{ad=n,d>0}\frac{1}{d^k}\sum_{b \,(\text{mod} \, d) }f\Big(\frac{az+b}{d}\Big).$$ Let $\mathcal{F}_k(N)$ be an orthonormal basis of $S_k(N)$ consisting only of joint eigenfunctions of the Hecke operators $T_n$. The Fourier expansion of $f$ at the cusp $\infty$ is given by $$f(z)=\sum_{n=1}^{\infty} a_n(f)n^{\frac{k-1}{2}}  e^{2\pi i nz}$$ for $f\in S_k(N).$  We denote $\lambda_n(f)$ to be the $n^{\text{th}}$ normalised Hecke eigenvalue of $f$, i.e. $T_n(f)=\lambda_n(f) f$. The Fourier coefficient $a_n(f)$ and $\lambda_n(f) $ are related by the condition  $a_n(f)=a_1(f)\lambda_n(f).$ From the Ramanujan-Deligne bound\cite{DP},  we have $$|\lambda_n(f)|\leq \tau(n),$$ where $\tau(n) $ denotes the divisor function. In particular $|\lambda_p(f)|\leq 2$ for a prime number $p.$

A sequence of real numbers $x_n \in [a,b]$ is equidistributed with respect to the probability measure $\mu$ if for any $[a',b']\subset [a,b],$ 
$$
\lim_{n\rightarrow\infty} \frac{|\{m\leq n\,:\, x_m\in[a',b'] \} |}{n}=\int_{a'}^{b'} d\mu.
$$  
Let $$\mu_p(x):=\frac{p+1}{\pi}\frac{\big(1-\frac{x^2}{4}\big)^{\frac{1}{2}}}{(\sqrt{p}+\sqrt{p^{-1}})^2-x^2} $$ and \begin{equation}\label{mukndefn}
\mu_{k,N}:= \frac{1}{\text{dim}(S_k(N))}\sum_{f\in \mathcal{F}_k(N)}\delta_{\lambda_p(f)} ,
\end{equation}
where $\delta_x$ is the Dirac measure at $x$ and $p$ is a prime number with $(p,N)=1.$
Using the Eichler-Selberg trace formula, for a fixed prime $p$, Serre \cite{JP} proved $\{\lambda_p(f)\,:\,  (p,N)=1$, $f \in \mathcal{F}_{k}(N)\}$ is equidistributed in $[-2,2]$ with respect to the measure $ \mu_p$  as $k+N\rightarrow \infty$. 
In other words, $\mu_{k,N}$ converges weakly to $\mu_p$  in $[-2,2],$ as $k+N\rightarrow \infty$ with $\gcd(p,N)=1.$ 
Given two probability measures $\mu_1$ and $\mu_2$ on a closed interval $\Omega \subset \mathbb{R},$ the discrepancy between $\mu_1$ and $\mu_2$ is given by 
$$ D(\mu_1,\mu_2) := \sup\{   |\mu_1(I)-\mu_2(I)|\, :\, I=[a,b]\subset \Omega \}.$$
The best known upper bound for $D(\mu_{k,N},\mu_p)$ is \begin{equation}\label{ubforDmu}
    D(\mu_{k,N},\mu_p)=O\Big((\log kN)^{-1}\Big),
    \end{equation}
    as proved by Murty and Sinha\cite{RK}.  
    
Let $S_k(N)^*$ be the space of primitive cusp forms with even integer weight $k$ and level $N$. Let $\tilde{T}_n$ be the restriction of the Hecke operator $T_n$ from $S_k(N)$ to its subspace $S_k(N)^*.$ Let 
$\mathcal{F}_k(N)^*$ be an orthonormal basis of $S_k(N)^*$ consisting only of joint eigenfunctions of the Hecke operators $\tilde{T}_n.$ 
We denote $\mu_{k,N}^*$ to be the corresponding measure associated to $\tilde{T}_p$ i.e. $$\mu_{k,N}^*:= \frac{1}{\text{dim}(S_k(N)^*)}\sum_{f\in \mathcal{F}_k(N)^*}\delta_{\lambda_p(f)}. $$
For a fixed prime $p$, it is not very difficult to deduce from Serre's theorem that 
$\mu^*_{k,N}$ converges weakly to $\mu_p$ as $k+N\rightarrow \infty$ with $\gcd(p,N)=1.$ 
An upper bound similar to equation \eqref{ubforDmu} can also be obtained for $D(\mu_{k,N}^*,\mu_p)$ (see \cite{MS2}). In recent work, Sarnak and Zubrilina \cite{SZ} give improved uniform estimates for $D(\mu_{2,N}^*,\mu_p)$ with the help of a new technique involving the use of the Petersson trace formula.

We have been discussing upper bounds for the discrepancies $D(\mu_{k,N}^*,\mu_p)$ and $D(\mu_{k,N},\mu_p)$.  A natural question that arises in this context is if one can find lower bounds or $\Omega$-type estimates for these discrepancies.
 For a fixed squarefree level N, Jung and Sardari \cite[Theorem 1.1]{JS} give us a sequence of weights $k_n$ with $k_n \rightarrow \infty  $ such that 
\begin{equation}\label{Discrepancych1}
 D(\mu_{k_n,N}^*,\mu_p)\gg \frac{1}{k_n^{\frac{1}{3}}\log^2 k_n}.     
\end{equation}

%Let $H_{k,N}^*$ denote a basis of arithmetically normalized Hecke eigenforms in the space orthogonal to oldforms. By arithmetically normalized we mean $a_1(f)=1$ for $f\in H_{k,N}^*$.
We now shift our focus to the distribution of $\{\lambda_{p^2}(f):f \in \mathcal{F}_k(N)^*\}$ which has been well investigated.
Let  
$$\mu_{p^2}(x)=\frac{p+1}{2 \pi}\frac{1}{(\sqrt{p}+\sqrt{p^{-1}})^2-(x+1)}\sqrt{\frac{3-x}{x+1}} $$
for $x\in[-1,3]$.
In 2009, Omar and Mazhouda \cite[Theorem 1]{OM} showed that $\{\lambda_{p^2}(f)\,:\, f\in \mathcal{F}_k(N)^*\}$ is equidistributed with respect to the measure $\mu_{p^2}(x)$  as $k\rightarrow \infty.$
Let $$\mu^*_{k,N,2}= \frac{1}{\text{dim}(S_k(N)^*)}\sum_{f\in \mathcal{F}_k(N)^*}\delta_{\lambda_{p^2}(f)} $$
For $N=1$, Tang and Wang \cite[Theorem 1]{TW} show the following analogue  of the equation \eqref{ubforDmu}  
\begin{equation}\label{ubforDmu2}
D(\mu^*_{k,1,2},\mu_{p^2})=\O\Big((\log k)^{-1}\Big).
\end{equation}

Let us consider the harmonic weight given by  $$\omega_{k,f}:=\frac{\Gamma(k-1)}{(4\pi)^{k-1}} |a_1(f)|^2$$ for $f\in  \mathcal{F}_k(N)$ (or $\mathcal{F}_k(N)^*$). The associated  weighted variant of the measure defined by equation \eqref{mukndefn} (see also equation \eqref{justdelta(1,n)}) is given by 
\begin{equation}\label{defmun}
\nu_{k,N}:=\sum_{f\in \mathcal{F}_k(N)}\omega_{k,f} \, \delta_{\lambda_p(f)}.
\end{equation}
We denote $\nu_{k,N}^*$ to be the corresponding weighted measure associated to $\tilde{\mathcal{T}}_p,$
i.e. $$\nu_{k,N}^*:=\sum_{f\in \mathcal{F}_k(N)^*} \omega_{k,f} \,\delta_{\lambda_p(f)}.$$
Let us consider the Sato-Tate measure defined  by
$$ \mu_\infty(x):=\frac{1}{\pi}\sqrt{1-\frac{x^2}{4}} 
$$
 for $x\in [-2,2].$ 
One can use the Petersson trace formula to show that both $\nu_{k,N}$ and $\nu_{k,N}^*$ converge weakly to $\mu_\infty $ as $k+N \rightarrow \infty $ with $\gcd(p,N)=1. $
For fixed squarefree levels, Theorem 1.6 of \cite{JS} gives us a sequence of weights $k_n$ with $k_n \rightarrow \infty  $ for which we have the following lower bound 
\begin{equation}\label{JSlb}
    D(\nu_{k_n,N}^*,\mu_\infty)\gg \frac{1}{k_n^{\frac{1}{3}}\log^2 k_n}. 
\end{equation}  
We obtain the following variant of  \eqref{JSlb} in the context of the space $S_k(N).$  
\begin{theorem}\label{Main theorem 1}
    Let $N=2^a b$ with $a=0,1,2$ and $b$  odd.  There exists an infinite sequence of weights $k_n$ with $k_n\rightarrow \infty $ such that $$
D(\nu_{k_n,N},\mu_\infty)\gg_N \frac{1}{k_n^{\frac{1}{3}} \log^2 k_n}. 
$$
\end{theorem}
One of the main ingredients in the proof of Theorem \ref{Main theorem 1} that enables us to generalize levels to the above form is the non-vanishing aspect of a certain Kloosterman sum mentioned in  Theorem \ref{asymptote}.

We now consider the measure 
$$\mu_{\infty,2} (x)=\frac{1}{2\pi} \sqrt{\frac{3-x}{1+x}} $$ for $x\in [-1,3]$ and an analogue of the measure $\nu_{k,N} $ given by 
$$\nu_{k,N,2}^:=\sum_{f\in \mathcal{F}_k(N)} \omega_{k,f} \, \delta_{\lambda_{p^2}(f)}.$$
As $p\rightarrow \infty,$ just as $\mu_p(x)\rightarrow\mu_\infty(x),$ we note that $\mu_{p^2}(x)\rightarrow \mu_{\infty,2} (x)$ (see \cite[Theorem 3]{OM}).
Using the Petersson trace formula, it can be shown that  $\nu_{k,N,2}$ converges weakly to $\mu_{\infty,2} $ as $k+N \rightarrow \infty $ with $\gcd(p,N)=1$ (see \cite[Theorem 3]{OM}). 
In this article, we provide a sequence of weights $k_n$ with $k_n \rightarrow \infty  $ for which a similar lower bound   as in equation \eqref{JSlb} for the discrepancy   $D(\nu_{k_n,N,2}, \mu_{\infty,2})$
 holds. More precisely, we have the following theorem. \begin{theorem}\label{Main theorem 2}
Let the level $N$ be fixed. Further let  $N=2^a b$ with $a=0,1,2$ and $b$  odd. Then there  exists a sequence of weights $k_n$ with $k_n\rightarrow \infty$ such that  $$D(\mu_{\infty,2},\nu_{k_n,N,2}) \gg_{N} \frac{1}{(\log k_n)^{3}(k_n)^{\frac{1}{3}}} . $$
\end{theorem}
 Again, the non-vanishing aspect of a certain Kloosterman sum is used to arrive at the proof of Theorem  \ref{Main theorem 2}.

\subsection*{Organization of the article}The structure of the paper is as follows. In Section \ref{asydel}, we derive an asymptotic for the Petersson trace formula under a given condition.  In Section \ref{NVKlsum}, we discuss the non-vanishing of certain classical Kloosterman sums. 
In Sections \ref{lampf} and \ref{lamp2f},
we discuss the proofs of  Theorems \ref{Main theorem 1} and \ref{Main theorem 2}
respectively.
\section{An asymptotic formula for $\Delta_{k,N}(m,n)$}\label{asydel}
In 1932, Petersson expressed  a weighted sum of $\overline{a_m(f)}a_n(f)$  over $f$ with $f\in \mathcal{F}_k(N)$ in terms of a Bessel function of the first kind and a Kloosterman sum (see 
 \cite{PH}). 
 For Hecke operators, this is the earliest weighted trace formula,
about 22 years earlier than the unweighted trace formula which was proved by Selberg\cite{Sel} in 1956.  We use the version of Petersson's trace formula given by \cite[Proposition 2.1]{ILS}.
Let  $$\rho_n(f)=\Bigg(\frac{\Gamma(k-1)}{ (4\pi )^{k-1}}\Bigg)^{\frac{1}{2}}   a_n(f)$$ and  
$$\Delta_{k,N}(m,n)=\sum_{f\in \mathcal{F}_k(N)}\overline{\rho_m(f)}\rho_n(f).$$
On taking   $m=1 $ in the above, we have 
\begin{align*}
   & \Delta_{k,N}(1,n)=\sum_{f\in \mathcal{F}_k(N)}\overline{\rho_1(f)}\rho_n(f) \\&
   =\sum_{f\in \mathcal{F}_k(N)} \left(\frac{\Gamma(k-1)}{ (4\pi )^{k-1}}\right) \overline{a_1(f)}a_n(f)=\sum_{f\in \mathcal{F}_k(N)} \left(\frac{\Gamma(k-1)}{ (4\pi )^{k-1}}\right) |a_1(f)|^2 \lambda_n(f), 
\end{align*}
since  $a_n(f)=a_1(f)\lambda_n(f)$.
Hence 
\begin{equation}\label{justdelta(1,n)}
\Delta_{k,N}(1,n)=\sum_{f\in \mathcal{F}_k(N)} \left(\frac{\Gamma(k-1)}{ (4\pi )^{k-1}}\right) |a_1(f)|^2\lambda_n(f).
\end{equation}
Equation \eqref{justdelta(1,n)} serves two purposes. 
It justifies the choice of variants $\nu_{k,N}$ and $\nu^*_{k,N}$ of $\mu_{k,N}$ and $\mu^*_{k,N}$ respectively.  Further, it indicates how Petersson's trace formula is indeed a weighted trace formula for Hecke operators.  Given integer $m,n$ and a natural number $c$, the Kloosterman sum $S(m,n,c)$  is defined to be $$S(m,n,c)=\sum_{x (\text{mod} \hspace{1 mm} c),\, \gcd(x,c)=1} e\left(\frac{mx+n\overline{x}(c)}{c}\right)$$    
where $\overline{x}(c)$ denotes the multiplicative inverse of $x$ 
  modulo $c$ and $e(x)=e^{2\pi ix}.$ 
  The Bessel functions of the first kind $J_a(x)$ for $a\geq 0$, can be defined by the following  power series representation 
$$
J_a(x)=\sum_{j=0}^{\infty } \frac{(-1)^j}{\Gamma(j+1)\Gamma(a+j+1)}\left(\frac{x}{2}\right)^{a+2j},
$$
where $\Gamma(x) $ denotes the Gamma function evaluated at $x$. 
 Now we recall Petersson's trace formula for $S_k(N).$
\begin{theorem}[{{\cite[Proposition\ 2.1]{ILS}}}] \label{PTf}
 Let $m,n, N$ be natural numbers and $k$ be an even natural number. We have $$\Delta_{k,N}(m,n)=\delta(m,n)+2\pi i^k \sum_{N|c,c>0} \frac{S(m,n;c)}{c}J_{k-1}\Big(\frac{4\pi\sqrt{mn}}{c}\Big).$$
\end{theorem}
\begin{lemma}[{{\cite[Section 2.1.1]{JS}}}] \label{Bessel}
We have the following estimates for the J-Bessel function.\\
(i) If $a \geq 0$ and $0<x\leq 1$, we have $$1\leq \frac{J_a(ax)}{x^aJ_a(a)}\leq e^{a(1-x)},$$
(ii)
$0<J_a(a) \ll \frac{1}{a^{\frac{1}{3}}}$ as $a\rightarrow \infty $,
\\
(iii) If $|d|<1$, then $$\frac{1}{a^{\frac{1}{3}}} \ll  J_a(a+da^{\frac{1}{3}})\ll \frac{1}{a^{\frac{1}{3}}}. 
$$
\end{lemma}
\begin{remark}
    It is worth noting that Lemma  \ref{Bessel}(iii) is important for finding a lower bound for Theorem \ref{asymptote}. 
The geometric origin of  Lemma  \ref{Bessel} (iii) has been discussed in \cite[Section 5]{JS}.
\end{remark}
Now we prove the following asymptotic of Petersson's trace formula that will be heavily used in future sections.  
\begin{theorem}\label{asymptote}
 Let $m,n, N$ be natural numbers and $k$ be an even natural number.  Let $N$ be fixed and  $m,n$ be such that $\left|\frac{4\pi\sqrt{mn}}{N}-(k-1)\right|<(k-1)^{\frac{1}{3}},$ then
\begin{align*}
      (i)\,\, \Delta_{k,N}(m,n)=\delta(m,n)+2\pi i^k \frac{S(m,n,N)}{N}J_{k-1}\left(\frac{4\pi\sqrt{mn}}{N} \right) +\O\left(\frac{e^{(k-1)\left(1-\frac{4}{9}-\log(\frac{9}{5})\right)}}{(k-1)^{\frac{1}{3}}}\right).
    \end{align*}
Furthermore, if $S(m,n,N)\neq 0$, then 
\begin{align*}
     (ii)\,\,\big| \Delta_{k,N}(m,n)-\delta(m,n)\big|\gg (k-1)^{-\frac{1}{3}}.
\end{align*}
\end{theorem}
\begin{proof}
Let us consider Petersson's trace formula 
 \begin{align*}
      \Delta_{k,N}(m,n)&=\delta(m,n)+2\pi i^k \sum_{N|c,c>0} \frac{S(m,n,c)}{c}J_{k-1}\Big(\frac{4\pi\sqrt{mn}}{c}\Big)
      \\
       &=\delta(m,n)+2\pi i^k\frac{S(m,n,N)}{N}J_{k-1}\Big(\frac{4\pi\sqrt{mn}}{N}\Big) +2\pi i ^k \sum_{b\geq 2} \frac{S(m,n,bN)}{bN}J_{k-1}\Big(\frac{4\pi\sqrt{mn}}{bN}\Big).
    \end{align*}
On using a trivial bound $|S(m,n,bN)|\leq \phi(bN), $ we obtain 
\begin{equation}\label{j}
\left|\sum_{b\geq 2}\frac{S(m,n,bN)}{bN}J_{k-1}\Big(\frac{4\pi\sqrt{mn}}{bN}\Big)\right|    \leq \sum_{b\geq 2}\left| J_{k-1}\Big(\frac{4\pi\sqrt{mn}}{bN}\Big)\right|.
\end{equation}
Now $\left|\frac{4\pi\sqrt{mn}}{N}-(k-1)\right|<(k-1)^{\frac{1}{3}}$ implies  \begin{equation}\label{k}
1-\frac{1}{(k-1)^{\frac{2}{3}}}<\frac{4\pi\sqrt{mn}}{(k-1)N}<1+\frac{1}{(k-1)^{\frac{2}{3}}}.
\end{equation}
Using Lemma \ref{Bessel}(i) we get 
\begin{align*}
     \left| J_{k-1}\left(\frac{4\pi\sqrt{mn}}{bN}\right)\right|&=\left| J_{k-1}\left((k-1)\frac{4\pi\sqrt{mn}}{(k-1)bN}\right)\right|
      \leq e^{a(1-x)}x^aJ_a(a)
    \end{align*}
where $a=k-1$ and $x=\frac{4\pi\sqrt{mn}}{(k-1)bN}.$
For $k>27, $
Equation \eqref{k}  and $b\geq 2$ yields 
$$\frac{8}{9b}<\frac{4\pi\sqrt{mn}}{(k-1)bN}<\frac{10}{9b}$$
which consequently  implies
\begin{align*}
    e^{a(1-x)}x^a&=e^{(k-1)(1-x+\log x)} \leq e^{(k-1)\big( 1-\frac{8}{9b}+\log \frac{10}{9b}\big)}.
\end{align*}
 Using Lemma \ref{Bessel}(ii) and proceeding  from equation \eqref{j}, we get 
\begin{align*}
 \sum_{b\geq 2}\Bigg| J_{k-1}\Big(\frac{4\pi\sqrt{mn}}{bN}\Big)\Bigg|
 \ll \frac{e^{(k-1)}}{(k-1)^{\frac{1}{3} }}\sum_{b\geq 2}e^{-(k-1)\big(\frac{8}{9b}+\log \frac{9b}{10}\big)}
\end{align*}
\begin{equation}\label{break}
    =   \frac{e^{(k-1)}}{(k-1)^{\frac{1}{3} }}\Bigg(\sum_{b=2}^4 e^{-(k-1)\big(\frac{8}{9b}+\log \frac{9b}{10}\big)} \Bigg)
 +\frac{e^{(k-1)}}{(k-1)^{\frac{1}{3} }}\Bigg(\sum_{b\geq 5}e^{-(k-1)\big(\frac{8}{9b}+\log \frac{9b}{10}\big)} \Bigg).  
\end{equation}
Since \begin{equation}\label{breakon}
\frac{e^{(k-1)}}{(k-1)^{\frac{1}{3} }}\Bigg(\sum_{b=2}^4 e^{-(k-1)\big(\frac{8}{9b}+\log \frac{9b}{10}\big)} \Bigg)=\O\left(\frac{e^{(k-1)\left(1-\frac{4}{9}-\log(\frac{9}{5})\right)}}{(k-1)^{\frac{1}{3}}}\right),
\end{equation} 
 we  estimate   $$\Bigg(\sum_{b\geq 5}e^{-(k-1)\big(\frac{8}{9b}+\log \frac{9b}{10}\big)} \Bigg).$$ Let $f(x)=e^{-(k-1)\left(\frac{8}{9x}+\log \frac{9x}{10}\right)}$, then   $f$ is decreasing for $x>2.$  By integral test, 
 \begin{equation}\label{itest}
 \Bigg(\sum_{b\geq 5}e^{-(k-1)\big(\frac{8}{9b}+\log \frac{9b}{10}\big)} \Bigg)\leq \int_4^\infty e^{-(k-1)\big(\frac{8}{9x}+\log \frac{9x}{10}\big)} \,dx. \end{equation}
Let $y=\frac{8}{9x}$ in the above integral so that $dy=-\frac{8}{9x^2}dx$ and $$ e^{-(k-1)\big(\frac{8}{9x}+\log \frac{9x}{10}\big)}=\frac{e^{-(k-1)y}}{(\frac{9x}{10})^{k-3}}\cdot \frac{100}{81x^2}=\frac{e^{-(k-1)y}}{(\frac{4}{5y})^{k-3}}\cdot \frac{100}{81x^2} .$$
Hence 
\begin{align*}
    \int_4^\infty e^{-(k-1)\big(\frac{8}{9x}+\log \frac{9x}{10}\big)} \, dx &= \frac{25}{18} \int_0^{\frac{2}{9} } e^{-(k-1)y}                  \left(\frac{5y}{4}\right)^{k-3}    \, dy 
    \\
    & \leq \frac{25}{18} \int_0^{\frac{2}{9} }  \left(\frac{5}{18} \right)^{k-3} e^{-(k-1)y}  \,dy 
    = \frac{25}{18}\left(\frac{5}{18} \right)^{k-3}\Bigg( \frac{e^{\frac{-2(k-1)} {9}} }{-(k-1)}+\frac{1}{(k-1)}\Bigg). 
\end{align*}
Using the above bound  for  equation \eqref{itest} with keeping equation \eqref{j} ,\eqref{break} and \eqref{breakon} in mind,  we have
$$\sum_{b\geq 2}\Bigg| J_{k-1}\Big(\frac{4\pi\sqrt{mn}}{bN}\Big)\Bigg|\ll  e^{(k-1)(1-\frac{4}{9}-\log(\frac{9}{5}))} + \frac{e^{(k-1)}}{(k-1)^{\frac{1}{3} }} \left(\frac{5}{18} \right)^{k-3}\left( \frac{e^{\frac{-2(k-1)}{9}} }{-(k-1)}+\frac{1}{(k-1)}\right).
$$
However,  $$
e^{(k-1)}\Big(\frac{5}{18} \Big)^{k-3}\Bigg( \frac{e^{\frac{-2(k-1)}{9}} }{-(k-1)}+\frac{1}{(k-1)}\Bigg)\leq \frac{2e^{2}}{(k-1)}\Big(\frac{5e}{18} \Big)^{k-3}=\o\left(e^{(k-1)\left(1-\frac{4}{9}-\log(\frac{9}{5})\right)}\right),$$
which implies 
\begin{equation}\label{finalineq}
\sum_{b\geq 2}\Bigg| J_{k-1}\Big(\frac{4\pi\sqrt{mn}}{bN}\Big)\Bigg|=
\O\left(\frac{e^{(k-1)\left(1-\frac{4}{9}-\log(\frac{9}{5})\right)}}{(k-1)^{\frac{1}{3}}}\right).
\end{equation}
On considering  equation \eqref{finalineq} and  equation \eqref{j}, 
$$ \hspace{2mm}\Delta_{k,N}(m,n)=\delta(m,n)+2\pi i^k \frac{S(m,n,N)}{N}J_{k-1}\left(\frac{4\pi\sqrt{mn}}{N} \right)+\O\left(\frac{e^{(k-1)\left(1-\frac{4}{9}-\log(\frac{9}{5})\right)}}{(k-1)^{\frac{1}{3}}}\right).
$$

 In order to show part(ii), let $\frac{4\pi\sqrt{mn}}{N}=(k-1)+d(k-1)^{\frac{1}{3}}$ with  $|d|<1$. On  applying Lemma \ref{Bessel} (iii) we have $$J_{k-1}\Big(\frac{4\pi\sqrt{mn}}{N} \Big)\gg \frac{1}{(k-1)^{\frac{1}{3}}}.$$
 Thus 
 \begin{align*}
     & \Delta_{k,N}(m,n)-\delta(m,n) =2\pi i^k \frac{S(m,n,N)}{N}J_{k-1}\left(\frac{4\pi\sqrt{mn}}{N} \right) +\O\left(\frac{e^{(k-1)\left(1-\frac{4}{9}-\log(\frac{9}{5})\right)}}{(k-1)^{\frac{1}{3}}}\right) \\
    & = 2\pi i^k \frac{S(m,n,N)}{N}J_{k-1}\left(\frac{4\pi\sqrt{mn}}{N} \right) + \o \left((k-1)^{-\frac{1}{3}}\right)   \gg (k-1)^{-\frac{1}{3}}.
 \end{align*}
\end{proof}
\begin{remark}[generalising Theorem \ref{asymptote}]
    Note that since $\left(1-\frac{4}{9}-\log(\frac{9}{5})\right)<-0.03$, we have $e^{(k-1)\left(1-\frac{4}{9}-\log(\frac{9}{5})\right)}=\o \left((k-1)^{-\frac{1}{3}}\right).$   When $S(m,n,N)= 0,$ we have $$\Delta_{k,N}(m,n)=\delta(m,n)+\O \left(\frac{e^{(k-1)\left(1-\frac{4}{9}-\log(\frac{9}{5})\right)}}{(k-1)^{\frac{1}{3}}}\right).$$
On modifying the proof of Theorem \ref{asymptote} a bit, we can get a slightly better bound of $\O \left(\frac{e^{\left(\frac{1}{2}-\log 2\right)(k-1)}}{(k-1)^{\frac{1}{3}}}\right).$ On considering the  equation \eqref{k},  as $k\rightarrow \infty,$ we can let  $1-\epsilon<\frac{4\pi\sqrt{mn}}{(k-1)N}<1+\epsilon$ in such a way that $\epsilon\rightarrow 0.$ If we observe the proof carefully and work with this $\epsilon$, the $\O $-term turns out to be $$\frac{e^{\left(1-\frac{(1-\epsilon)}{2}+\log \left(\frac{(1+\epsilon)}{2}\right) \right) (k-1)}}{(k-1)^{\frac{1}{3}}}.$$ The error term in the above theorem has been worked out with $\epsilon=\frac{1}{9}.$ However, there is a limit to which we can optimize the big $\O $ estimate by this method. The optimal upper bound thus possible by the above proof  is  $\O \left(\frac{e^{(k-1)\left(\frac{1}{2}-\log 2\right)}}{(k-1)^{\frac{1}{3}}}\right)$ which is because  $$\lim_{\epsilon\rightarrow 0^+}1-\frac{(1-\epsilon)}{2}+\log \left(\frac{(1+\epsilon)}{2}\right)=\frac{1}{2}-\log 2.$$
\end{remark}
\begin{remark}
 For $N=1$,  Theorem \ref{asymptote} gives a better error term as compared to $\O\left(k^{-\frac{1}{2}}\right)$ of  \cite[Theorem 1.7]{JS}. More precisely, 
$$
\Delta_{k,1}(m,n)=\delta(m,n)+2\pi i^k J_{k-1}\Big(4\pi\sqrt{mn}\Big)+\O\left(\frac{e^{(k-1)(1-\frac{4}{9}-\log(\frac{9}{5}))}}{(k-1)^{\frac{1}{3}}}\right).
$$
 Theorem \ref{asymptote} also expands on the remark of  \cite[Theorem 1.7]{JS}. For a fixed level $N,$ corollary 2.3 of \cite{ILS} gives an asymptotic for $\Delta_{k,N}(m,n)$ with the condition   $\frac{4\pi\sqrt{mn}}{N}\leq \frac{k}{3},$ whereas the above theorem gives asymptotic for  $(k-1)-(k-1)^{\frac{1}{3}}<\frac{4\pi\sqrt{mn}}{N}< (k-1)+(k-1)^{\frac{1}{3}} $ for a fixed level $N$. Whenever $\left|\frac{4\pi\sqrt{mn}}{N}-(k-1)\right|<(k-1)^{\frac{1}{3}}$ and  $S(m,n,N)\ne 0$, we have   $$\big| \Delta_{k,N}(m,n)-\delta(m,n)\big|\gg  (k-1)^{-\frac{1}{3}}. $$ 
\end{remark}
\begin{remark}[comparing Theorem \ref{asymptote} with Theorem 1.7 of \cite{JS}]
    In this remark, we compare two asymptotic at hand. As proved in \cite[Theorem 1.7]{JS}, for the condition $|4\pi\sqrt{mn}-k| <2k^{\frac{1}{3}}$ and $\gcd(mn,N)=1,$ the main term for $\Delta^*_{k,N}(m,n)$ is $2\pi i^{-k}\frac{\mu(N)}{N} \prod_{p\mid N}  (1-p^{-2})J_{k-1}(4\pi\sqrt{mn})$ for squarefree level $N.$ Theorem \ref{asymptote} gives us an asymptotic for $\Delta_{k,N}(m,n)$ with the condition  
$\left|\frac{4\pi\sqrt{mn}}{N}-(k-1)\right|<(k-1)^{\frac{1}{3}}$ for any level $N,$ provided   $S(m,n,N)\ne 0.$  For a purpose of application  to Theorem \ref{asymptote}, we show the non-vanishing of the Kloosterman sum $S(1,p^{2n},N)$  in Lemma \ref{main lemma}  for any natural number $n$ and level $N$ not divisible by $8.$
    
\end{remark}
\section{Non-vanishing of certain classical Kloosterman sums}\label{NVKlsum}
Recall that the Kloosterman sum $S(m,n,c)$ is given by $$S(m,n,c)=\sum_{x (\text{mod} \hspace{1 mm} c),\, \gcd(x,c)=1} e^{2\pi i\left(\frac{mx+n\overline{x}(c)}{c}\right)},$$   
where   $m,n$ are integers, $c$ is a natural number and $\overline{x}(c)$ denotes the multiplicative inverse of $x$   modulo $c$. 
In 1926,  Kloosterman sums were introduced as an application to solve the following problem. Given natural numbers $a,b,c,d$ and $n$, Kloosterman obtained an asymptotic formula in \cite{Kloosterman} for the number of representations of a natural number $n$
in the form $ax^2+by^2+cz^2+ dt^2.$  The sums for the special case of  $m=0 $ or $n=0 $  are called Ramanujan sums.  For squarefree $N$, the main term in \cite[Theorem 1.7]{JS} is $$2\pi i^{-k}\frac{\mu(N)}{N} \prod_{p\mid N}  (1-p^{-2})J_{k-1}(4\pi\sqrt{mn}). $$ Note that $\mu(N)=S(1,0,N)$ (see equation (3.4) of \cite{IK}) which is nonzero if and only if $N$ is squarefree. Thus in order to have a Kloosterman sum as the main term, we need to ensure the nonvanishing of that particular Kloosterman sum (see Theorem \ref{asymptote}(ii)).  

Throughout this section, we assume $N$ to be  coprime to prime $p.$ 
For $a,b\in \mathbb{N} , $ $(a,b)=1, $ let $ \overline{a}(b)$ denote the multiplicative inverse of $a$ in $(\mathbb{Z}/ b\mathbb{Z})^*$ i.e.  $a \overline{a}(b)\equiv 1$ (mod $b$). 
Now we demonstrate a lemma regarding the multiplicative property of Kloosterman sums. The lemma is particularly useful to show the non-vanishing of a Kloosterman sum by reducing the task to show the nonvanishing for prime powers. 
\begin{lemma}\label{multiplicativity}
Let $N=q_1^{\beta_1}  \,\dots\,  q_t^{\beta_t},$
where $q_1,  \,\dots\, , q_t$ are distinct prime numbers,  $q_i^{\beta_i}=l_i$ and   $$c_i=\frac{N}{\prod_{j=1}^i l_j}$$  for $1\leq i\leq t.$ 
Further let $m_0=1,$ $m_i=\prod_{j=1}^i\overline{l_j}({c_{j}})$ for $1\leq i\leq t-1.$ 
For a positive integer $n,$ we have 
\begin{equation}\label{prodKlo}
S(1,p^{2n},N)=\prod_{i=1}^t S(m_{i-1}\overline{c_i}({l_i}), m_{i-1}\overline{c_i}({l_i})p^{2n},l_i).
\end{equation}
\end{lemma}
\begin{proof}
The main idea of the proof is to 
 use the multiplicative of the Kloosterman sum (see equation (1.59) of \cite{IK}) repetitively. The property is given by 
\begin{equation}\label{mul.prop.}
S(a,b,cd)=S(a\overline{c}(d),b\overline{c}(d),d)S(a\overline{d}(c),b\overline{d}(c),c)
\end{equation}
where $(c,d)=1$.
Let us consider the last two terms in the product given by equation \eqref{prodKlo} where the index $i$ is written in increasing order,i.e.,  
\begin{align*}
&S(m_{t-1},m_{t-1}p^{2n},l_t)S(m_{t-2}\overline{c_{t-1}}({l_{t-1}}),m_{t-2}\overline{c_{t-1}}({l_{t-1}})p^{2n},l_{t-1}) \\
=&S(m_{t-2}\overline{l_{t-1}}({c_{t-1}}),m_{t-2}\overline{l_{t-1}}({c_{t-1}})p^{2n},l_t)S(m_{t-2}\overline{c_{t-1}}({l_{t-1}}),m_{t-2}\overline{c_{t-1}}({l_{t-1}})p^{2n},l_{t-1}).
\end{align*}
 By equation \eqref{mul.prop.}, the above product is equal to $$S(m_{t-2},m_{t-2}p^{2n},l_{t-1}l_t)$$ since $\overline{c_{t-1}}({l_{t-1}})l_t=\overline{l_t}({l_{t-1}})l_t=1$ (mod $l_{t-1}$).  In a  similar fashion, we note that 
 \begin{align*}
    & S(m_{t-i},m_{t-i}p^{2n},l_{t-i+1}\, ...  \,l_t)S(m_{t-i-1}\overline{c_{t-i}}({l_{t-i}}),m_{t-i-1}\overline{c_{t-i}}({l_{t-i}})p^{2n},l_{t-i})\\
     =&S(m_{t-i-1},m_{t-i-1}p^{2n},l_{t-i} l_{t-i+1}\, ...     \,l_t)
 \end{align*}
 for $2\leq i\leq t.$
Above equality is valid because  $m_{t-i}=m_{t-i-1}\overline{l_{t-i}}({c_{t-i}})$ and 
 \begin{align*}
    &\overline{c_{t-i}}({l_{t-i}})l_{t-i+1}\,... \,l_t=\overline{\Bigg(\frac{N}{l_1 \,... \,  l_{t-i}}\Bigg)}({l_{t-i}})  l_{t-i+1}\,... \,l_t \\
    =&\overline{l_{t-i+1}\,... \,l_t}({l_{t-i}})  l_{t-i+1}\,... \,l_t=1
 \end{align*}
 modulo $l_{t-i}$.
 This proves 
$$S(1,p^{2n},N)=\prod_{i=1}^t S(m_{i-1}\overline{c_i}({l_i}), m_{i-1}\overline{c_i}({l_i})p^{2n},l_i).$$
\end{proof}
\begin{remark}
A similar result like Lemma \ref{multiplicativity} for $S(a,b,N)$ with $a,b\in \N$
can be derived as well. However, keeping equation \eqref{lb} in mind, we consider only  $S(1,p^{2n}, N).$ 
\end{remark}
In the next three lemmas, we show the non-vanishing of Kloosterman sum $S(1,p^{2n},N)$ for some specific values of $N$ leading to  $N$ not divisible by $8.$  This is crucial so that we can apply Theorem \ref{asymptote}.
\begin{lemma}\label{odd prime power}
Let $N=q_1^{\beta_1}  \,\dots\,  q_t^{\beta_t},$
where $q_1,  \,\dots\, , q_t$ are distinct odd prime numbers.
Let $\beta_1,\,...\, ,\beta_t$ be such that $\beta_i\geq 2.$   Then $S(1,p^{2n},N)\neq 0$ for all $n\in \mathbb{N}.$
\end{lemma}
\begin{proof}
By virtue of Lemma \ref{multiplicativity}, it is enough to show that for all $i,$  $$S(m_{i-1}\overline{c_i}({l_i}), m_{i-1}\overline{c_i}({l_i})p^{2n},l_i) \neq 0,$$ where 
$l_i=q_i^{\beta_i}$.
Let $i$ be fixed. 
Firstly, we show that $q_i$ does not divide $2(m_{i-1}\overline{c_i}({l_i}))^2p^{2n}$ using method of contradiction. Suppose  
$q_i|2(m_{i-1}\overline{c_i}({l_i}))^2p^{2n},$  so that we get   $q_i$ divides $(m_{i-1}\overline{c_i}({l_i}))^2.$ This implies 
$q_i|m_{i-1}^2,$ since $q_i|(\overline{c_i}({l_i}))^2$ would imply $l_i$ should divide $c_i^{2\beta_i}(\overline{c_i}({l_i}))^{2\beta_i}=1$ which is an impossibility. 
Now $q_i$ being a prime number must divide $m_{i-1}$
 where $m_{i-1}=\prod_{j=1}^{i-1}\overline{l_j}({c_{j}})$.   Let the  equation $(\tilde{j})$ for this particular proof be given by $l_j\overline{l_j}({c_j})\equiv 1$ (mod $c_j).$ Further let equation  $(j)'$ for this particular proof be 
$$
l_j\overline{l_j}({c_j})l_2 \, \dots\,   l_j\equiv l_2 \, \dots\,    l_j \, (\textrm{mod} \,\,  c_1).
$$
Considering the equations $(\tilde{1}),(2)', \, \dots\,   , (i-1)'$, we get 
$$
l_1\, \dots\,  l_{i-1} \overline{l_1}({c_1})\, \dots\,   \overline{l_{i-1}}({c_{i-1}}) l_2^{i-2} l_3^{i-3}\, \dots\,   l_{i-1}\equiv
l_2^{i-2} l_3^{i-3}\, \dots\,   l_{i-1}\, ( \textrm{mod} \,\,  c_1).
$$
Since $q_i|c_1$ for $i>1,$ $q_i$ divides $$\left(l_1\, \dots\,  l_{i-1} \overline{l_1}({c_1})\, \dots\,   \overline{l_{i-1}}({c_{i-1}})l_2^{i-2} l_3^{i-3}\, \dots\,   l_{i-1}\right)-\left(
l_2^{i-2} l_3^{i-3}\, \dots \,    l_{i-1}\right).$$
However by our supposition of $q_i|m_{i-1},$  we obtain  $q_i|l_2^{i-2} l_3^{i-3}\, \dots \,   \, l_{i-1}$, which is absurdity. 

 We have shown that $S(m_{i-1}\overline{c_i}({l_i}), m_{i-1}\overline{c_i}({l_i})p^{2n},l_i)$ satisfy the required hypothesis for exercise 1 of  \cite[Chapter 12]{IK}. On letting 
$\Re(z)$ denote the real part of the complex number $z$, we have  
\begin{equation}\label{klformula}
S(m_{i-1}\overline{c_i}({l_i}), m_{i-1}\overline{c_i}({l_i})p^{2n},l_i)=2\Bigg( \frac{l'}{l_i}\Bigg)\sqrt{l_i} \Re \Big(\epsilon_{l_i} e^{\frac{4\pi i l'}{l_i}} \Big)
\end{equation}
where $l'^2\equiv (m_{i-1}\overline{c_i}({l_i})p^n)^2$ (mod $l_i$), $\epsilon_{l_i}=1$ or $i$ as per $l_i\equiv 1 $ or $l_i\equiv 3$ (mod $4$) respectively (see \cite[equation (2.2)]{IK}). We now claim that $l'\notin \{\frac{\alpha l_i}{8}\, |\,  \alpha \in \mathbb{Z}\}$ using method of contradiction again. Suppose not, then $l'=\frac{\alpha l_i}{8}$ for some $\alpha\in \mathbb{Z}.$ The congruence satisfied by $l'$ implies  
$$\alpha^2l_i^2-64l_it=64(m_{i-1}\overline{c_i}({l_i})p^n)^2$$
for some $ t\in \mathbb{Z}.$
 Hence $q_i$ must divide $64(m_{i-1}\overline{c_i}({l_i})p^n)^2,$ an absurdity as shown earlier in this proof.  
 Thus  $S(1,p^{2n},N)\neq 0$.
\end{proof}
One of the key steps in the above lemma is the closed formula given by equation \eqref{klformula}. The Kloosterman sum $S(a,b,l_i)$ to modulus $l_i=q_i^{\beta_i}$ with $\beta_i\geq 2$ were first computed by Salie\cite{SAL}. We note that a closed formula for prime modulus $q$, i.e. for $S(a,b,q)$ is not known. Hence for nonvanishing for prime modulus, we proceed as given by Lemma \ref{odd prime}. 
\begin{lemma}\label{odd prime}
Let $q$ be an odd prime number. Then $S(1,a,q)\neq 0$ for $a\in \mathbb{Z}.$
\end{lemma}
\begin{proof}
Let $\Phi_q(x)=1+x+x^2+\dots+x^{q-1}$
so that
$\mathbb{Z}[x]/(\Phi_q(x),q)\cong \mathbb{Z}[x]/\left((x-1)^q,q\right)$ (see exercise 8 of  \cite[Section 13.6]{DF}).
Further, $ \mathbb{Z}\Big[e^{\frac{2\pi i}{q}}\Big]$ is isomorphic to $\mathbb{Z}[x]/(\Phi_q(x))$ where  $e^{\frac{2\pi i}{q}}$ gets mapped to $x+(\Phi_q(x))$.  Consider the ring homomorphism  $$
\mathbb{Z}[x]/(\Phi_q(x))\rightarrow \mathbb{Z}[x]/\left((x-1)^q,q\right) \rightarrow \mathbb{Z}[x]/\left((x-1),q\right)\rightarrow \mathbb{F}_q
$$
where $$x+(\Phi_q(x))\mapsto x+\left((x-1)^q,q\right) \mapsto x+((x-1),q)$$ $$=1+\left((x-1),q\right)\mapsto 1.$$
Hence $e^{\frac{2\pi i}{q}}$ gets mapped to 1 via the ring homomorphism. However $$S(1,a,q)=\sum_{x=1}^{q-1}\left(e^{\frac{2\pi i }{q}}\right)^{(x+a\overline{x}(q))}$$ implies $S(1,a,q)$ gets mapped to $q-1$ in $\mathbb{F}_q.$ This infact shows that $S(1,a,q)$ can not be $0$. This is because on supposing  $S(1,a,q)=0$  makes the image of $S(1,a,q)$ under the ring homomorphism  to be $0$ in $\mathbb{F}_q.$
\end{proof}
\begin{lemma}\label{main lemma}
Let $N=2^a b$ with $a=0,1,2$ and $b$ be an odd number. Then $S(1,p^{2n},N)\neq 0$ for all $n\in \mathbb{N}.$
\end{lemma}
\begin{proof}
By   Lemma \ref{multiplicativity}, we have 
$$S(1,p^{2n},N)=
\prod_{i=1}^t S(m_{i-1}\overline{c_i}({l_i}), m_{i-1}\overline{c_i}({l_i})p^{2n},l_i)
=S(1,(m_{i-1}\overline{c_i}({l_i})p^{n})^2,l_i).$$  Lemma \ref{odd prime power} and Lemma \ref{odd prime} imply if $l_i$ is odd, then $S(1,(m_{i-1}\overline{c_i}({l_i})p^{n})^2,l_i)\neq 0$.    
Now 
$S(1,a^2,2^b)=S(1,1,2^b)$ for $a$ odd and $b=1,2,3.$  Since  $m_{i-1}\overline{c_i}({l_i})p^{n}$ is odd and  the values of $S(1,1,2), S(1,1,4)$ and $S(1,1,8)$ are  $1,-2$ and $0$ respectively, we 
obtain $S(1,(m_{i-1}\overline{c_i}({l_i})p^{n})^2,l_i)\neq 0$ for  $l_i=2,4.$ 
\end{proof}
\begin{remark}
    One can compute $S(1,p^{2n},2^a)$ for $a > 3$, $n\in \N$ and check for the nonvanishing  of these. However, we do it till $a=3$ as $S(1,1,8)$ vanishes.
\end{remark}
\section{Lower bound for  $D(\nu_{k,N},\mu_\infty)$ with levels not divisible by 8}\label{lampf}
The Ramanujan-Deligne bound\cite{DP} implies $|\lambda_p(f)|\leq 2$ and we can write $\lambda_p(f)=2\cos (\theta_p(f))$ for a unique $\theta_p(f)\in[0,\pi].$ Let $X_n(2\cos \theta)=\frac{\sin(n+1)\theta}{\sin \theta}$ for $n\in \mathbb{N}\cup \{0 \}$ be the $n$th Chebyshev polynomial of the second kind. By Lemma 3 of 
\cite{CDF},  $\lambda_{p^n}(f)=X_n(\lambda_p(f)). $
For $n\geq 1,$ orthogonality of Chebyshev polynomials yield \begin{equation}\label{ortho}
    \int_{-2}^2X_n(x) \,  d\mu_\infty(x)=\int_{-2}^2 X_0(x)X_n(x) \,  d\mu_\infty(x)=0.
\end{equation} 
Recall that  $$\nu_{k,N}
=\sum_{f\in \mathcal{F}_k(N)}   \omega_{k,f} \,\delta_{\lambda_p(f)}.$$
On taking  $k_n=\Big[\frac{4\pi p^n}{N} \Big]$ and using  Theorem \ref{asymptote}, we have 
\begin{equation}\label{lb}
  \int_{-2}^2X_{2n}(x) \,  d\nu_{k_n,N}=\Delta_{k_n,N}(1,p^{2n})=2\pi i^k \frac{S(1,p^{2n},N)}{N}J_{k_n-1}\left(\frac{4\pi p^n}{N} \right)+\o\left((k_n-1)^{-\frac{1}{3}}\right).   
\end{equation}
\begin{proof}[Proof of Theorem \ref{Main theorem 1}]
On letting $k_n=\Big[\frac{4\pi p^{n}}{N}\Big]$ 
 and applying Theorem \ref{asymptote}(ii), Lemma \ref{main lemma}, equation \eqref{ortho} and equation \eqref{lb}, 
\begin{align*}
    \int_{-2}^2 X_{2n}(x) \,  d(\nu_{k_n,N}-\mu_\infty)(x)  &=\int_{-2}^2 X_{2n}(x) \,  d\nu_{k_n,N}(x)  -  \int_{-2}^2X_{2n}(x) \,  d\mu_\infty(x)\\
    &=\int_{-2}^2 X_{2n}(x) \,  d\nu_{k_n,N}(x)\gg_N (k_n-1)^{-\frac{1}{3}}. 
\end{align*}
Using integration by parts, we have
\begin{align*}
  &  \Bigg|\int_{-2}^2 X_{2n}(x) \,  d(\nu_{k_n,N}-\mu_\infty)(x)\Bigg| \\
    &=
\Bigg|\Big[X_{2n}(x) \,  (\nu_{k_n,N}-\mu_\infty)(x)\Big]_{-2}^2  -  \int_{-2}^2 X'_{2n}(x) \,  (\nu_{k_n,N}-\mu_\infty)(x) dx\Bigg| \\
 &\leq  \Bigg|\big(X_{2n}(2)-X_{2n}(-2)\big) (\nu_{k_n,N}-\mu_\infty)([-2,2]) \Bigg|
+\Bigg|\int_{-2}^2 X'_{2n}(x) \,  (\nu_{k_n,N}-\mu_\infty)(x) dx\Bigg| \\
&\ll   n^2\Bigg|   (\nu_{k_n,N}-\mu_\infty)([-2,2]) \Bigg|,
\end{align*}
where the last inequality follows from the fact that 
 $\Big| X'_{2n}(x) \Big|\ll n^2$ and $|X_{2n}(2)|=|X_{2n}(-2)|=2n+1.$
Hence
\begin{align*}
    & D(\nu_{k_n,N},\mu_\infty)\geq \left| \nu_{k_n,N}([-2,2])-\mu_\infty([-2,2])  \right|
 =\left|   (\nu_{k_n,N}-\mu_\infty)([-2,2]) \right| \\
 & \gg\frac{1}{n^2} \left|\int_{-2}^2 X_{2n}(x) \,  d(\nu_{k_n,N}-\mu_\infty)(x)\right|\gg_N n^{-2}(k_n-1)^{-\frac{1}{3}}. 
\end{align*}
However, as $k_n=\left[\frac{4\pi p^{n}}{N}\right]$ we have $\log k_n \gg_N n$
and 
 $$D(\nu_{k_n,N},\mu_\infty)\gg_N\frac{1}{(k_n-1)^{\frac{1}{3}}(\log k_n)^2}.
 $$
\end{proof}
\begin{remark}
Theorem 1.6 of \cite{JS} gives us a sequence of weights for squarefree levels such that a lower bound like equation \eqref{Discrepancych1} holds. Theorem \ref{Main theorem 1} gives an analogue of  Theorem 1.6 of \cite{JS} in the context of cusp forms to more levels of the form $2^ab$ with $b$ odd and $a=0,1,2$.  Note that the natural density of squarefree integers is $\frac{6}{\pi^2}\approx 0.6079,$ whereas the natural density of the levels included in Theorem \ref{Main theorem 1}  is $0.875$. It must be noted that we can still include level $N$ which are divisible by 8 in the assumption of Theorem \ref{Main theorem 1}, provided for a given prime $p,$ there exists a sequence $y_n$ such that $S(1,p^{2y_n},N)\neq 0$ for all $y_n.$
\end{remark}
\section{Lower bound for  $D(\mu_{\infty,2},\nu_{k,N,2})$ with levels not divisible by 8}\label{lamp2f}
We start with the generating function for $X_m(x)$ which is given by 
$$
\sum_{m=0}^\infty X_m(x)t^m=\frac{1}{1-tx+t^2}. 
$$
Consider  the generating function for $Y_m(x)$ which is given by  
$$
\sum_{m=0}^\infty Y_m(x)t^m=\frac{1+t}{1-(x-1)t+t^2}.
$$ For a positive integer $m,$ on comparing generating functions, we can show 
\begin{equation}\label{YX2}
X_{2m}=Y_m\circ X_2.\end{equation}
  Let $$Q_{2m+1}=Y_1+Y_3+\dots +Y_{2m+1}$$ and 
$$Q_{2m}=Y_0+Y_2+\dots +Y_{2m}$$
 for a positive integer $m.$
 We have $\lambda_{p^m}(\text{Sym}^2 f)=Q_m\left(\lambda_{p}(\text{Sym}^2 f)\right)$   for  $f \in S_k(N)$  \cite[Lemma 4]{OM}.
From equation (4.2) of \cite{OM},
\begin{equation}\label{intQn}
\int_{-1}^3 Q_{2n+1}(t) \,  d\mu_{\infty,2}(t)=\int_{-2}^2 \left( Q_{2n+1}\circ X_2 \right)(t)\,  d\mu_{\infty}(t)
=\delta(2n+1,\text{even})=0.
\end{equation}

\begin{lemma}\label{bound for Q'}
Let  $n$ be a natural number.  Then  $|Q'_{n}(x)|\ll n^{3} $ and 
 $|Q_{n}(3)|\ll n^2 $ for $x\in [-1,3].$
\end{lemma}
\begin{proof}
We prove for $n$ even and the case of $n$ being odd is similar.  Consider $x_0,y_0$ such that $X_2(y_0)=y_0^2-1=x_0$.
 The fact  $|X'_n(x)|\ll n^2$ for all $x\in [-2,2]$ 
 implies
$$\left|\sum_{i=0}^n X'_{4i}(y_0)\right|\ll n^3.
$$
Using $Y_n\circ X_2=X_{2n},$  we get 
\begin{align*}
  \Big|Q'_{2n}(x_0)\Big|\ll n^3.
\end{align*}
Furthermore,  $|Q_{2n}(3)|\ll n^2$ as $X_{2m}(2)=2m+1$ and $$Q_{2n}(3)=\sum_{i=0}^n Y_{2i}(3)=\sum_{i=0}^n X_{4i}(2).$$
\end{proof}
\begin{lemma}\label{asymptote 2}
Let $N=2^a b$ with $a=0,1,2$ and $b$ be an odd number. For a natural number $n$ and a sequence $k_n$  satisfying  $$\left|\frac{4\pi p^{2n+1}}{N}-(k_n-1)\right|<(k_n-1)^{\frac{1}{3}},$$ we have 
\begin{align*}
   & (i) \sum_{i=0}^{n-1} \Delta_{k_n,N}(1,p^{4i+2})=\O_N\Bigg(\frac{\log k_n}{(k_n-1)^{\frac{1}{3}}} \cdot \Big(\frac{5e}{18}\Big)^{k_n-1}\Bigg),\\&
   (ii) \Delta_{k_n,N}(1,p^{4n+2})\gg_N (k_n-1)^{-\frac{1}{3}}.
\end{align*}
\end{lemma}
\begin{proof}
Let $1\leq i\leq n-1$ be fixed. Using Petersson's trace formula we get 
$$\Delta_{k_n,N}(1,p^{4i+2})=2\pi i^k \sum_{N|c,c>0} \frac{S(1,p^{4i+2};c)}{c}J_{k_n-1}\left(\frac{4\pi p^{2i+1}}{c}\right).$$
Note that $$\frac{4\pi p^{2i+1}}{(k_n-1)N}<\frac{1}{p^2}\left(1+ (k_n-1)^{-\frac{2}{3}}\right)$$ which implies $\frac{4\pi p^{2i+1}}{(k_n-1)N}<\frac{1}{4}\cdot\frac{10}{9}=\frac{5}{18}$. Now using Lemma \ref{Bessel} we get,
\begin{align*}
  &\Big|\Delta_{k_n,N}(1,p^{4i+2})\Big|\leq \left|J_{k_n-1}\left(\frac{4\pi p^{2i+1}}{N}\right)\right|+ \sum_{b=2}^\infty \left| J_{k_n-1}\left(\frac{4\pi p^{2i+1}}{bN}\right)\right|\\&
  \leq e^{(k_n-1)}J_{k_n-1}(k_n-1)\left(\left(\frac{5}{18}\right)^{k_n-1}+\sum_{b=2}^\infty \left(\frac{5}{18b}\right)^{k_n-1}\right)\\&
  \ll_N \left(\frac{5e}{18}\right)^{k_n-1}(k_n-1)^{-\frac{1}{3}}\Bigg[1+\int_1^\infty \left(\frac{1}{x}\right)^{(k_n-1)} \, dx\Bigg]\ll_N \left(\frac{5e}{18}\right)^{k_n-1}(k_n-1)^{-\frac{1}{3}}.
\end{align*}
Thus
\begin{align*}
    &\left|\sum_{i=0}^{n-1} \Delta_{k_n,N}(1,p^{4i+2})\right|\leq \sum_{i=0}^{n-1} \Big|\Delta_{k_n,N}(1,p^{4i+2})\Big| \\& \ll_N n \, \left(\frac{5e}{18}\right)^{k_n-1}(k_n-1)^{-\frac{1}{3}}\ll_N \log k_n\left(\frac{5e}{18}\right)^{k_n-1}(k_n-1)^{-\frac{1}{3}}. 
\end{align*}
By Theorem \ref{asymptote}(ii), $$\left|\Delta_{k_n,N}(1,p^{4n+2})\right|\gg_N(k_n-1)^{-\frac{1}{3}}.$$
\end{proof}
\begin{remark}
    Note that in order to prove Lemma \ref{asymptote 2}(ii) we use Theorem \ref{asymptote}(ii). However, we can not use Theorem \ref{asymptote}(i) directly to deduce Lemma \ref{asymptote 2}(i). This is because in the summation   $$ \sum_{i=0}^{n-1} \Delta_{k_n,N}(1,p^{4i+2}),$$ for instance, the term  $\Delta_{k_n,N}(1,p^{2})$ does not satisfy the assumption of Theorem \ref{asymptote}(i) for a fixed $N$ and a very large value of $k_n$.
\end{remark}
Recall that  $$\nu_{k,N,2}=
\sum_{f\in \mathcal{F}_k(N)}\omega_{k,f} \,\delta_{\lambda_{p^2}(f)}.$$
\begin{lemma}\label{asymptote 3}
Let $N=2^a b$ with $a=0,1,2$ and $b$ be an odd number. There exists a sequence $k_n\rightarrow \infty$ such that 
$$\Bigg|\int_{-1}^3 (Q_{2n+1})(t) \,  d\nu_{k_n,N,2} (t)\Bigg| \gg_{N} \frac{1}{(k_n-1)^{\frac{1}{3}}} .$$
\end{lemma}
\begin{proof}
Consider 
\begin{align*}
  &  \int_{-1}^3 Q_{2n+1}(t) \,  d\nu_{k_n,N,2}(t)
    \\& = \frac{\Gamma(k-1)}{(4\pi)^{k-1}   }\sum_{f\in \mathcal{F}_k(N)} |a_1(f)|^2 Q_{2n+1}
    (\lambda_{p^2}(f)) 
     \\& = \frac{\Gamma(k-1)}{(4\pi)^{k-1}   }\sum_{f\in \mathcal{F}_k(N)} |a_1(f)|^2 Q_{2n+1}(X_2(\lambda_{p}(f))
     \\&=\int_{-2}^2 Q_{2n+1}\circ X_2(t) \,  d\nu_{k_n,N}(t)=\sum_{i=0}^n \frac{\Gamma(k-1)}{(4\pi)^{k-1}   }\sum_{f\in \mathcal{F}_k(N)} |a_1(f)|^2\lambda_{p^{4n+2}}(f)
     \\&  =\sum_{i=0}^n \Delta_{k,N}(1,p^{4n+2}).
\end{align*}
On taking  $k_n=\left[ \frac{4\pi p^{2n+1}}{N}\right],$  the proof is complete by Lemma \ref{asymptote 2} and noting that $$\log k_n\left(\frac{5e}{18}\right)^{k_n-1}=\o_N\left(1 \right).$$
\end{proof}
\begin{proof}[Proof of Theorem \ref{Main theorem 2}]
The proof is similar to that of Theorem \ref{Main theorem 1}. Using equation \eqref{intQn} and Lemma \ref{asymptote 3}, we have  
\begin{equation}\label{intQ>>-1/3}
\int_{-1}^3 Q_{2n+1}(x) \,  d(\nu_{k_n,N,2}-\mu_{\infty,2})(x)  
=\int_{-1}^3 Q_{2n+1}(x) \,  d\nu_{k_n,N,2}(x)\gg_N (k_n-1)^{-\frac{1}{3}}. \end{equation}
 Integration by parts implies
 \begin{align*}
     & \Bigg|\int_{-1}^3 Q_{2n+1}(x) \,  d(\nu_{k_n,N,2}-\mu_{\infty,2})(x)\Bigg|
     \\&\leq \Bigg|\big(Q_{2n+1}(-1)-Q_{2n+1}(3)\big)(\nu_{k_n,N,2}-\mu_{\infty,2})([-1,3])\Bigg|
+\Bigg|\int_{-1}^3 Q'_{2n+1}(x) \,  (\nu_{k_n,N,2}-\mu_{\infty,2})(x) dx\Bigg|.    
 \end{align*}
By Lemma \ref{bound for Q'}  the above expression is 
$$
\ll_\epsilon n^{3+\epsilon}\Big|  (\nu_{k_n,N,2}-\mu_{\infty,2})([-1,3]) \Big|.
$$
Hence 
\begin{align*}
    & D(\nu_{k_n,N,2},\mu_{\infty,2})\geq \Big| \nu_{k_n,N,2}([-1,3])-\mu_{\infty,2}([-1,3])  \Big|
 =\Big|  (\nu_{k_n,N,2}-\mu_{\infty,2})([-1,3]) \Big|\\&
 \gg\frac{1}{n^{3}} \Bigg|\int_{-2}^2 Q_{2n+1}(x) \,  d(\nu_{k_n,N,2}-\mu_{\infty,2})(x)\Bigg|\gg_N n^{-(3)} (k_n-1)^{-\frac{1}{3}}
\end{align*}
where the last step follows from equation \eqref{intQ>>-1/3}.
For the sequence $k_n=\left[ \frac{4\pi p^{2n+1}}{N}\right]$, indeed, 
$$D(\nu_{k_n,N,2},\mu_{\infty,2})\gg_{N} \frac{1}{(\log k_n)^{3}(k_n-1)^{\frac{1}{3}} } .$$
\end{proof}
\begin{remark}
Theorem \ref{Main theorem 2} gives an analogue of  \cite[Theorem 1.6]{JS}  in the context of the distribution of $\lambda_{p^2}(f)$ with levels not divisible by $8$ and the space of cusp forms. 
\end{remark}
In this section, some crucial ingredients in proving Theorem \ref{Main theorem 2} is use of  equations \eqref{YX2}, \eqref{intQn}   and the formula $\lambda_{p^m}(\text{Sym}^2 f)=Q_m\left(\lambda_{p}(\text{Sym}^2 f)\right)$.  However, while analysing the distribution of $\lambda_{p^3}(f)$, it becomes difficult to get a polynomial $Q_{3,m}(x)$  such that $Q_{3,m}\left(\lambda_p(\text{Sym}^3 f)\right)=\lambda_{p^{m}}(\text{Sym}^3 f).$ Another challenge is to get applicable versions of equations  \eqref{YX2} and \eqref{intQn}  while considering the distribution of $\lambda_{p^3}(f).$ 
 These have been some of the obstacles in getting 
 an analogue of Theorem \ref{Main theorem 1} and \ref{Main theorem 2}
for  $\lambda_{p^3}(f).$ 
The level of difficulty increases as we move for the distribution of higher powers like $\lambda_{p^4}(f)$ or $\lambda_{p^5}(f).$
\subsection*{Acknowledgements}
The author is supported by a fellowship from CSIR. The author is thankful to Kaneenika Sinha for introducing the problem in the first place. The author is grateful to his Ph.D. supervisors Baskar Balasubramanyam and Kaneenika Sinha for their valuable insights while the work was being done. 
\bibliographystyle{alpha}
\bibliography{Sato-TateNT.bib}

\end{document}